\documentclass{amsart}
\usepackage{graphicx}
\usepackage{amscd}
\usepackage{amsmath}
\usepackage{amsfonts}
\usepackage{amssymb}
\theoremstyle{plain}
\newtheorem{theorem}{Theorem}
\newtheorem{corollary}[theorem]{Corollary}
\newtheorem{lemma}[theorem]{Lemma}
\newtheorem{proposition}[theorem]{Proposition}
\theoremstyle{definition}

\newtheorem{remark}[theorem]{Remark}
\theoremstyle{remark}

\newcommand{\Spec}{\mbox{Spec}\ \! }

\begin{document}

\title{A note on perinormal domains}

\author{Tiberiu Dumitrescu and Anam Rani}

\address{Facultatea de Matematica si Informatica,University of Bucharest,14 A\-ca\-de\-mi\-ei Str., Bucharest, RO 010014, Romania}
\email{tiberiu@fmi.unibuc.ro, tiberiu\_dumitrescu2003@yahoo.com}

\address{Abdus Salam School of Mathematical Sciences, GC University,
Lahore. 68-B, New Muslim Town, Lahore 54600, Pakistan} \email{anamrane@gmail.com}

\begin{abstract}
\noindent
Recently,  N. Epstein and J. Shapiro introduced and studied the  perinormal domains: those domains $A$ whose going down overrings are flat $A$-modules. We show that every Pr\"ufer v-multiplication domain is perinormal and has no proper lying over overrings.
We also show that a treed  perinormal domain is a Pr\"ufer domain. 
We give two pull-back constructions that produce perinormal/non-perinormal domains.
\end{abstract}



\maketitle

\section{Introduction}

In their recent paper \cite{ES}, N. Epstein and J. Shapiro introduced and studied the  {\em perinormal domains}: those domains $A$ whose going down overrings are flat $A$-modules. Due to \cite[Theorem 2]{R}, $A$ is perinormal if and only if whenever $B$ is a local overring of $A$, we have $B=A_{N\cap A}$ where $N$ is the maximal ideal of $A$. Here, as usual, a {\em going down overring} $B$ of $A$ means a ring between $A$ and its quotient field such that going down holds for $A\subseteq B$, that is, the induced spectral map $\Spec(B_Q)\rightarrow \Spec(A_{Q\cap A})$ is surjective for each $Q\in \Spec(B)$. The perinormality is a local property cf. \cite[Theorem 2.3]{ES}. 
Also studied in \cite{ES}  is the subclass of  {\em globally perinormal domains}: those domains $A$ whose going down overrings are fraction rings of $A$. So $A$ is globally perinormal if and only if $A$ is  perinormal and its flat overrings are fraction rings of $A$.

The purpose of this note is to extend some of the results in \cite{ES}. 
We obtain the following results (all needed definitions, though standard, are recalled in the next section at the place where they are used).
An essential domain has no proper lying over overrings (Theorem \ref{1}). Every P-domain \cite{MZ}  is perinormal (Theorem \ref{3}).
Consequently, a Pr\"ufer v-multiplication domain (PVMD)
is perinormal and has no proper lying over overrings (Corollary \ref{5}).
In particular, we retrieve \cite[Theorem 7.5]{ES} which states that a generalized Krull domain is perinormal (Corollary \ref{5}).
Using Heinzer's example \cite{H}, we remark that an essential domain is not necessarily perinormal (Remark \ref{8}).
A treed  perinormal domain is a Pr\"ufer domain (Proposition \ref{2}). We thus extend   \cite[Proposition 3.2]{ES} which says  that a one dimensional perinormal domain is a Pr\"ufer domain.
The pullback construction in \cite[Theorem 5.2]{ES}   produces perinormal domains starting from semilocal Krull domains.  
We extend this construction  relaxing the Krull domain hypothesis to P-domain and removing the semilocal restriction (Theorem \ref{11}). We also give a construction producing non-perinormal domains (Theorem \ref{12}).
In \cite[Theorem 6.3]{ES}, it was shown that a Krull domain with torsion divisor class group is globally perinormal.
We extend this result by showing that a PVMD  with torsion class group (e.g a GCD domain) is globally perinormal (Theorem \ref{6}). Consequently, an AGCD domain  is globally  perinormal if and only if it is integrally closed (Corollary \ref{13}). In Theorem \ref{14}, we slightly improve Theorem \ref{6}.

Throughout this paper all rings are (commutative unitary) integral domains. Any unexplained terminology is standard like in \cite{G}, \cite{Ha} or \cite{LM}.
\\[2mm]




\section{Results}

Let $A$ be a domain.
Call a prime ideal $P$ of $A$ a {\em valued prime}  if $A_P$ is a valuation domain.
Recall that  $A$ is said to be an {\em essential domain} if 
$A=\cap_{P\in G}A_P$ where $G$ is the set of valued primes of $A$.
As usual, an overring $B$ of $A$ is a {\em lying over overring} if the map $\Spec(B)\rightarrow \Spec(A)$ is surjective.


\begin{theorem}\label{1}
If $A$ is an essential domain and $B$ is a lying over overring of $A$, then $A=B$.
\end{theorem}
\begin{proof}
 Since $A$ is essential, $A=\cap_{P\in G}A_P$ where $G$ is the set of valued primes of $A$.
As  $A\subseteq B$ satisfies lying over, 
for each $P\in G$, we can choose $P'\in \Spec(B)$ such that $P'\cap A=P$. Since $A_P$ is a valuation domain, it follows that
 $A_P=B_{P'}$, cf. \cite[Theorem 26.1]{G}. Then $A\subseteq B\subseteq \cap_{P\in G}B_{P'}=\cap_{P\in G}A_P=A$.
So $B=A$.
\end{proof}

Recall \cite{MZ} that a domain $A$ is called a {\em P-domain} if
$A_P$ is a valuation domain for every  prime ideal  $P$ which is minimal over an ideal of the form $Aa:b$ with $a,b\in A$. 
A P-domain is an essential domain but not conversely, cf. \cite[Proposition 1.1]{MZ} and \cite{H}.
Moreover,  a fraction ring of a P-domain is still a P-domain, cf. \cite[Corollary 1.2]{MZ}.
We state the main result of this paper.
\begin{theorem}\label{3}
Every P-domain  is perinormal.
\end{theorem}
\begin{proof}
Let $A$ be a P-domain and $B$ a going down overring of $A$. Let $Q\in \Spec(B)$ and $P=Q\cap A$. By \cite[Proposition 1.1]{MZ} and \cite[Lemma 2.2]{ES}, $A_P$ is an essential domain. By going down, $B_Q$ is a lying over overring of $A_P$. So Theorem \ref{1}  applies to give $A_P=B_Q$. Thus $A\subseteq B$ is flat due to \cite[Theorem 2]{R}.
\end{proof}

\begin{remark}\label{8}
An essential domain is not necessarily perinormal. 
Indeed, the essential domain $D$ constructed in \cite{H}  has a one dimensional localization $D_P$ which is not a valuation domain. Then $D$ is not perinormal  cf. \cite[Proposition 3.2]{ES}.
\end{remark}

\begin{remark}\label{4}
The underlying idea of the proof of Theorem \ref{3} is very simple.
Let $\mathcal{D}$ a class of domains which is closed under localizations at prime ideals. If every $A\in \mathcal{D}$ has no proper lying over overring, then every $A\in \mathcal{D}$ is perinormal. Indeed, if $B$ is  a going down overring of $A$, $Q\in \Spec(B)$ and $P=Q\cap A$, then $A_P\subseteq B_Q$ satisfies lying over, so $A_P=B_Q$. 
\end{remark}

Let $A$ be a domain with quotient field $K$.
Recall that  $A$ is a {\em Pr\"ufer v-multiplication domain (PVMD)} if for every  finitely generated nonzero ideal $I$, there exists a finitely generated nonzero ideal $J$ such that $(IJ)_v$ is a principal ideal.
Here, as usual, for a fractional nonzero ideal $H$ of $A$, 
its divisorial closure ($v$-closure) is the fractional  ideal $H_v:=(H^{-1})^{-1}$ where $H^{-1}$ is fractional  ideal $A:H=\{x\in K\mid xH\subseteq A\}$. 
By \cite[Corollary 1.4 and Example 2.1]{MZ}, a PVMD is a P-domain but not conversely. 
It is well-known that a GCD domain is a PVMD, (see for instance \cite[Proposition 6.1]{MZ}). From these remarks and Theorem \ref{1}, we have


\begin{corollary}\label{5}
A PVMD (e.g. a GCD domain) is perinormal and  has no proper lying over overrings. 
\end{corollary}

Let $A$ be a domain and $X^{1}(A)$ be the set of height one prime ideals of $A$. 
Recall \cite[section 43]{G} that   $A$ is a {\em generalized Krull domain} (resp. {\em Krull domain})  if $A=\cap_{P\in X^{1}(A)} A_P$, this intersection has finite character and $A_P$ is a valuation domain (resp. a discrete valuation domain) for each $P\in X^{1}(A)$. Clearly a Krull domain is a  generalized Krull domain.
Since a generalized Krull domain is a PVMD (cf. \cite[Theorem 7]{Gr}), we retrive 

\begin{corollary}{\em(\cite[Theorems 3.10 and 7.5]{ES})}\label{9}
A (generalized) Krull domain is perinormal.
\end{corollary}

\begin{remark}
Recall that a {\em FC (finite conductor) domain} is a domain in which every intersection of two principal ideals is finitely generated. By \cite[Theorem 2]{M}, an integrally closed FC  domain has no proper lying over overrings. As both normality and FC condition localize, we derive that an integrally closed FC domain is perinormal, cf. Remark \ref{4}. But this is in fact  a consequence of Corollary \ref{5} because an integrally closed FC domain is a PVMD, cf. \cite[Corollary 2.5]{FZ}.
\end{remark}

Recall that a {\em treed domain} is a domain  whose incomparable prime ideals are comaximal.  As their localizations at the prime ideals are valuation domains, the Pr\"ufer domains are treed and perinormal. We prove the converse.

\begin{proposition}\label{2}
A treed  perinormal domain is a Pr\"ufer domain.
\end{proposition}
\begin{proof}
Localizing, we  may assume that $A$ is local with maximal ideal $P$. Then $\Spec(A)$ is linearly ordered. By \cite{KO}, there exist a lying over valuation overring $V$ of $A$. Then $V$ is a going down overring of $A$. Since $A$ is  perinormal, it follows that  $A=V$ because  the maximal ideal of $V$ lies over $P$. 
\end{proof}

The   assertion $(a)$ below extends \cite[Proposition 4.4]{MZ} while assertion $(b)$ is essentially \cite[Proposition 3.2]{ES}.

\begin{corollary}\label{10}
$(a)$ A  treed  P-domain  is a Pr\"ufer domain. $(b)$ A one dimensional perinormal domain is a Pr\"ufer domain.
\end{corollary}
\begin{proof}
Combine Theorem \ref{3}, Proposition \ref{2} and the fact that a one dimensional  domain is treed.
\end{proof}

Our next result extends the pullback construction in \cite[Theorem 5.2]{ES} producing perinormal domains. 
While our proof uses the same idea, we relax the Krull domain hypothesis to P-domain and remove the semilocal restriction.

Let $B$ be a P-domain and $M_1,...,M_n$ maximal ideals of $B$ such that none of them is a valued prime. Assume further that 
all fields $B/M_i$ are isomorphic to the same field $K$
by isomorphisms   $\sigma_i:B/M_i\rightarrow K$, $i=1,...,n$.  
Set $I=M_1\cap\cdots\cap M_n$.
Let $\pi:B\rightarrow K^n$ be the composition  of the  canonical morphism $B\rightarrow B/I$, the Chinese Remainder Theorem morphism $B/I\rightarrow B/M_1\times\cdots\times B/M_n$ and $(\sigma_1,...,\sigma_n):B/M_1\times\cdots\times B/M_n\rightarrow K^n$. Finally, identify $K$ with its   diagonal image  in $K^n$.

\begin{theorem}\label{11}
In the setup above, the pulback domain $A=\pi^{-1}(K)$ is perinormal. 
\end{theorem}
\begin{proof}
Note that $B$ is perinormal by Theorem \ref{3}.
Clearly $I=\ker(\pi)$ is a common ideal of $A$ and $B$.
Let $(C,N)$ be a going down overring of $A$ and set $P=N\cap A$. Assume that  $P\not\supseteq I$. By usual pullback arguments we have $C\supseteq A_P=B_Q$ where $Q=PA_P\cap B$. Moreover since $C$ is a going down extension of the perinormal domain $B_Q$ and $N\cap B_Q=QB_Q$, we get  $A_P=B_Q=C$, so we are done in this case. Assume now  that  $P\supseteq I$. By \cite[Lemma 1.1.6]{FHP}, we may localize $A$ and $B$ in $A-P$ and thus assume that $A$ is local with maximal ideal $P$. Then $A\subseteq C$ satisfies lying over because it satisfies going down and $P\subseteq N$. Let $G$ be the set of valued primes of $B$. For every $H\in G$, select $H'\in \Spec(C)$ such that $H'\cap A=H\cap A$. As none of  $M_1,...,M_n$ is a valued prime, we get $H\not\supseteq I$, so $B_H=A_{H\cap A}=C_{H'}$ because $B_H$ is a valuation domain.
We have
$A\subseteq C\subseteq \cap_{H\in G}C_{H'}=\cap_{H\in G}B_H=B$ because $B$ is an essential domain, so $A\subseteq C\subseteq B$.
We claim that $A=C$. Indeed, as  $I$ is a common of $A,B,C$, we get $A/I=K\subseteq C/I\subseteq B/I=K^n$. Since $C/I$ is local and $K$ is the only local ring between $K$ and $K^n$, we derive that $C/I=K$, so $C=A$. 
\end{proof}

The following pullback construction provides examples of non-perinormal domains. 

\begin{theorem}\label{12}
Let $B$ be a domain, $M$ a maximal ideal of $B$, $\pi:B\rightarrow B/M$  the  canonical map  and $K$ a proper subfield of $B/M$. Then the pulback domain $A=\pi^{-1}(K)$ is not perinormal.
\end{theorem}
\begin{proof}
Clearly $M$ is a maximal ideal of both $A$ and $B$. 
We claim that  going down holds for $A\subseteq B$. Let $Q\in \Spec(B)$ and $P=Q\cap A$. It suffices to prove that the map $\Spec(B_Q)\rightarrow \Spec(A_P)$ is surjective. If $Q\neq M$ and $f\in M-Q$, then $A_f=B_f$, so $A_P=(A_f)_{PA_f}=(B_f)_{QB_f}=B_Q$. Assume that $Q=M$, so $P=M$. Let $N\in \Spec(A)$ be a proper subideal of $M$. As done above, we get $A_N=B_H$ where $H=NA_N\cap B$. In particular, $H$ lies over $N$. We show that $H\subseteq M$. If not, select $g\in H-M$ and $h\in B$ such that $\pi(gh)=1$. Then $gh\in N-M$, a contradiction. It remains that $H\subseteq M$, so going down holds for $A\subseteq B$. But $A_M\neq B_M$, because $A_M/MA_M=K$ which is a proper subfied of $B/M$. So $A$ is not perinormal.
\end{proof}

Let $A$ be a domain.
Recall (\cite{S} and \cite{Z1}), that  $A$ is called an {\em almost GCD domain} or {\em AGCD} if for each $x,y\in A$, there exists $n\geq 1$ such that $x^nA\cap y^nA$ is a principal ideal.

Recall that the class group of a domain was introduced by Bouvier and Zafrullah (see \cite{B} and \cite{BZ})  in order to extend the divisor class group concept from the Krull domains case to arbitrary domains. For simplicity, we choose to recall this definition  only in the PVMD case. 
Let $A$ be a PVMD. The set $D(A)=\{ H_v\mid H$ finitely generated nonzero fractional ideal of $A\}$ is a group  under the operation $(H_1,H_2)\mapsto (H_1H_2)_v$ called the $v$-multiplication. The {\em class group} $Cl(A)$ of $A$ is defined as $D(A)$ modulo the subgroup of all principal nonzero fractional ideals. 
By \cite[Corollary 1.5]{BZ}, the GCD domains are exactly the PVMDs with zero class group. The next lemma collects some  known facts.

\begin{lemma}{\em (\cite{Z1} and \cite{AZ})}\label{7}

$(a)$ Let $A$ be an AGCD domain and $A'$ its integral closure.  Then $A\subseteq A'$ is a root extension (that is, every $x\in A'$ has some power in $A$).

$(b)$ The PVMDs with torsion class group are exactly the integrally closed AGCD domains.

$(c)$ Let $A$ be an AGCD domain. Then every flat overring of $A$ is a fraction ring of $A$.
\end{lemma}
\begin{proof}
Part $(a)$ is \cite[Theorem 3.1]{Z1}, part $(b)$ is  \cite[Theorem 3.9]{Z1}, part $(c)$ is \cite[Theorem 3.5]{AZ}.
\end{proof}

According to \cite{ES},  a domain $A$ is called {\em globally perinormal} if every going down overring of $A$ is a fraction ring of $A$. In \cite[Theorem 6.3]{ES}, it was shown that a Krull domain with torsion divisor class group is globally perinormal. In Theorem \ref{6} we extend this result to PVMDs. 

\begin{theorem}\label{6}
If $A$ is a PVMD  with torsion class group (e.g a GCD domain), then $A$  is globally perinormal.
\end{theorem}
\begin{proof}
By Corollary \ref{5}, $A$ is perinormal. Combine parts $(b)$ and $(c)$ of Lemma \ref{7} to complete the proof.
\end{proof}

\begin{corollary}\label{13}
For an AGCD domain $A$, the following assertions are equivalent.

$(a)$ $A$ is globally  perinormal,

$(b)$ $A$ is perinormal,

$(c)$ $A$ is integrally closed.
\end{corollary}
\begin{proof}
$(a)\Rightarrow (b)$ is clear. $(b)\Rightarrow (c)$ Let $A'$ 
be the  integral closure  of $A$. By part $(a)$ of Lemma \ref{7}, it follows that $A\subseteq A'$ is a root extension. By \cite[Theorem 2.1]{AZ}, the natural map $\Spec(A')\rightarrow \Spec(A)$ is an order isomorphism, hence $A\subseteq A'$ satisfies going down. Since $A$ is perinormal, it follows that $A\subseteq A'$ is flat, so $A=A'$ cf. \cite[Proposition 2]{R}.
$(c)\Rightarrow (a)$ Apply part $(b)$ of Lemma \ref{7} and Proposition \ref{6}.
\end{proof}

In the last result of this paper we slightly improve Theorem \ref{6}
(note that the Pr\"ufer domain case of Theorem \ref{14} is \cite[Theorem 27.5]{G}).  Let $A$ be a domain. An  overring $B$ of $A$ is called {\em t-linked} (over $A$) \cite{DHLZ} if, for each finitely generated nonzero ideal $I$  of $A$ such that $I^{-1}=A$, one has $(IB)^{-1}=B$. By \cite[Proposition 2.2]{DHLZ}, every flat overring of $A$ is t-linked over $A$. So a perinormal domain whose t-linked overrings are fraction rings of $A$ is globally perinormal.
 
\begin{theorem}\label{14}
Assume that $A$ is a PVMD which satisfies the following condition:
 for each finitely generated nonzero ideal $I$  of $A$, we have $I^n\subseteq bA\subseteq I_v$ for some $n\geq 1$ and  $b\in A$. Then $A$ is  globally perinormal.
\end{theorem}
\begin{proof}
By \cite[Theorem 1.3]{DHLZ1},  every t-linked overring of $A$ is a fraction ring of $A$. As a PVMD is perinormal (due to Corollary \ref{5}), the paragraph preceding this theorem applies to complete the proof.
\end{proof}

{\bf Acknowledgements.} 
We thank N. Epstein and J. Shapiro for sending us the latest version of their paper \cite{ES}.
The first author gratefully acknowledges the warm
hospitality of the Abdus Salam School of Mathematical Sciences GC University Lahore during his many visits in the period 2006-2015. 
The second author is highly grateful to ASSMS GC University Lahore, Pakistan in supporting and facilitating this researh.
\\[3mm]

{}

\end{document}